\documentclass[11 point, a4paper, reqno]{amsart}
\usepackage[english]{babel}
\usepackage{amsthm}
\usepackage{graphicx}
\usepackage{amsfonts}
\usepackage{eufrak}
\usepackage{amsopn}
\usepackage{amsmath, amsfonts, amssymb, amscd, amsthm}
\usepackage{yhmath}
\usepackage{setspace}
\usepackage{mathtools}
\usepackage{multirow}
\usepackage{caption}
\usepackage{graphicx}
\usepackage{lscape}
\usepackage[all]{xy}

\newtheorem{theorem}{Theorem}[section]
\newtheorem{lemma}[theorem]{Lemma}
\newtheorem{proposition}[theorem]{Proposition}
\newtheorem{corollary}[theorem]{Corollary}

\theoremstyle{definition}

\theoremstyle{remark}
\newtheorem{remark}[theorem]{Remark}
\theoremstyle{example}

\theoremstyle{note}

\numberwithin{equation}{section}

\DeclareMathOperator{\Hom}{Hom}

\DeclareMathOperator{\GL}{GL}

\DeclareMathOperator{\Ker}{Ker}

\DeclareMathOperator{\Image}{Im}
\DeclareMathOperator{\Span}{Span}
\DeclareMathOperator{\ind}{ind}
\DeclareMathOperator{\M}{M}

\DeclareMathOperator{\tr}{Tr}

\DeclareMathOperator{\Rank}{Rank}

\begin{document}
\title{On a twisted Jacquet module of $\GL(2n)$ over a finite field}
\author{Kumar Balasubramanian$^{*}$} \thanks{* indicates corresponding author}
\thanks{Research of Kumar Balasubramanian is supported by the SERB grant: MTR/2019/000358.}
\address{Kumar Balasubramanian\\
Department of Mathematics\\
IISER Bhopal\\
Bhopal, Madhya Pradesh 462066, India}
\email{bkumar@iiserb.ac.in}

\author{Abhishek Dangodara}
\address{Abhishek Dangodara\\
Department of Mathematics\\
IISER Bhopal\\
Bhopal, Madhya Pradesh 462066, India}
\email{abhishek17@iiserb.ac.in}

\author{Himanshi Khurana}
\address{Himanshi Khurana\\
Department of Mathematics\\
IISER Bhopal\\
Bhopal, Madhya Pradesh 462066, India}
\email{himanshi18@iiserb.ac.in}

\keywords{Cuspidal representations, Twisted Jacquet module}
\subjclass{Primary: 20G40}

\maketitle

\begin{abstract} Let $F$ be a finite field and $G=\GL(2n,F)$. In this paper, we explicitly describe a certain twisted Jacquet module of an irreducible cuspidal representation of $G$. 
\end{abstract}

\section{Introduction}

Let $F$ be a finite field and $G=\GL(n,F)$. Let $P$ be a parabolic subgroup of $G$ with Levi decomposition $P=MN$. Let $\pi$ be any irreducible finite dimensional complex representation of $G$ and $\psi$ be an irreducible representation of $N$. Let $\pi_{N, \psi}$ be the sum of all irreducible representations of $N$ inside $\pi$, on which $\pi$ acts via the character $\psi$. It is easy to see that $\pi_{N, \psi}$ is a representation of the subgroup $M_{\psi}$ of $M$, consisting of those elements in $M$ which leave the isomorphism class of $\psi$ invariant under the inner conjugation action of $M$ on $N$. The space $\pi_{N, \psi}$ is called the \textit{twisted Jacquet module} of the representation $\pi$. It is an interesting question to understand for which irreducible representations $\pi$, the twisted Jacquet module $\pi_{N, \psi}$ is non-zero and to understand its structure as a module for $M_{\psi}$. \\ 

In \cite{KumHim},\cite{KumHimGL6}, inspired by the work of Prasad in \cite{Dip[1]}, we studied the structure of a certain twisted Jacquet module of a cuspidal representation of $\GL(4,F)$ and $\GL(6,F)$. Based on our calculations, we had conjectured the structure of the module for $\GL(2n,F)$ (see Section 1 in \cite{KumHimGL6}). For a more detailed introduction and the motivation to study the problem, we refer the reader to Section 1 in \cite{KumHim}.\\

Before we state our result, we set up some notation. Let $F$ be a finite field and $F_n$ be the unique field extension of $F$ of degree $n$. Let $G=\GL(2n,F)$ and $P=MN$ be the standard maximal parabolic subgroup of $G$ corresponding to the partition $(n,n)$. We have, $M \simeq \GL(n,F) \times \GL(n,F)$ and $N \simeq \M(n,F)$. We let $\pi=\pi_{\theta}$ to be an irreducible cuspidal representation of $G$ associated to the regular character $\theta$. Let $\psi$ be any character of $N \simeq \M(n,F)$ and $\psi_0$ be a fixed non-trivial character of $F$. We let 
\[ A_i =\begin{bmatrix} I_i & 0\\
0 & 0 \end{bmatrix} \in \M(n,F).\]
Let $\psi_A : N \to \mathbb{C}^{\times}$ be the character given by 
\[\psi_A \left ( \begin{bmatrix} 1 & X\\
0 & 1 \end{bmatrix} \right)=\psi_0(\tr(AX)).\]

Let $H_A=M_1 \times M_2$ where $M_1$ is the Mirabolic subgroup of $\GL(n,F)$ and $M_2=w_0{M_1}^T{w_0}^{-1}$ where 
\[ w_0 =\begin{bmatrix} 0 & \hdots & 1\\
\vdots & \adots & \vdots \\
1 & \hdots & 0
\end{bmatrix}\]
Let $U$ be the subgroup of unipotent matrices in $\GL(2n,F)$ and $U_A=U \cap H_A$. Then, we get $U_A \simeq U_1 \times U_2$ where $U_1$ and $U_2$ are the upper triangular unipotent subgroups of $\GL(n,F)$. For $k=1,2$, let $\mu_k : U_k \to \mathbb{C}^{\times}$ be the non-degenerate character of $U_k$ given by 
\[\mu_{k} \left ( \begin{bmatrix}
 1 & x_{12} & x_{13} & \hdots & x_{1,n}\\
  &   1   & x_{23} & \hdots & x_{2,n} \\
 & & 1 & \ddots & \vdots \\
 & & & \ddots & x_{(n-1),n}\\
 & & & & 1 \end{bmatrix} \right )= \psi_0(x_{12}+ x_{23}+ \cdots + x_{(n-1),n}).\]
Let $\mu: U_{A}\rightarrow \mathbb{C}^{\times}$ be the character of $U_{A}$ given by \[\mu(u)=\mu_{1}(u_{1})\mu_{2}(u_{2})\] where $u= \begin{bmatrix} u_{1} & 0 \\ 0 & u_{2}\end{bmatrix} \in U_A$. \\

\begin{theorem}
Let $\theta$ be a regular character of $F_n^{\times}$ and $\pi=\pi_{\theta}$ be an irreducible cuspidal representation of $G$. Then 
\[\pi_{N,\psi_A} \simeq \theta|_{F^{\times}} \otimes \ind_{U_A}^{H_A}\mu\]
as $M_{\psi_A}$ modules.

\end{theorem}

\section{Preliminaries}
In this section, we mention some preliminary results that we need in our paper.

\subsection{Character of a Cuspidal Representation} Let $F$ be the finite field of order $q$ and $G=\GL(m,F)$. Let $F_m$ be the unique field extension of $F$ of degree $m$. A character $\theta$ of $F^{\times}_{m}$ is called a ``regular'' character, if under the action of the Galois group of $F_{m}$ over $F$, $\theta$ gives rise to $m$ distinct characters of $F^{\times}_{m}$. It is a well known fact that the cuspidal representations of $\GL(m,F)$ are parametrized by the regular characters of $F_{m}^{\times}$. To avoid introducing more notation, we mention below only the relevant statements on computing the character values that we have used. We refer the reader to Section 6 in \cite{Gel[1]} for more precise statements on computing character values.

\begin{theorem}\label{character value of cuspidal representation (Gelfand)}
Let $\theta$ be a regular character of $F^{\times}_{m}$. Let $\pi=\pi_{\theta}$ be an irreducible cuspidal representation of $\GL(m,F)$ associated to $\theta$. Let $\Theta_{\theta}$ be its character. If $g\in \GL(m,F)$ is such that the characteristic polynomial of $g$ is not a power of a polynomial irreducible over $F$. Then, we have  \[\Theta_\theta(g)=0. \]
\end{theorem}

\begin{theorem}\label{character value of cuspidal representation (Dipendra)} Let $\theta$ be a regular character of $F^{\times}_{m}$. Let $\pi=\pi_{\theta}$ be an irreducible cuspidal representation of $\GL(m,F)$ associated to $\theta$. Let $\Theta_{\theta}$ be its character. Suppose that $g=s.u$ is the Jordan decomposition of an element $g$ in $\GL(m,F)$. If $\Theta_{\theta}(g)\neq 0$, then the semisimple element $s$ must come from $F_{m}^{\times}$. Suppose that $s$ comes from $F_{m}^{\times}$. Let $z$ be an eigenvalue of $s$ in $F_{m}$ and let $t$ be the dimension of the kernel of $g-z$ over $F_{m}$. Then
\[\Theta_{\theta}(g)=(-1)^{m-1}\bigg[\sum_{\alpha=0}^{d-1}\theta(z^{q^{\alpha}})\bigg ](1-q^{d})(1-(q^{d})^{2})\cdots (1-(q^{d})^{t-1}). \]
where $q^{d}$ is the cardinality of the field generated by $z$ over $F$, and the summation is over the distinct Galois conjugates of $z$.
\end{theorem}

See Theorem 2 in \cite{Dip[1]} for this version.

\subsection{Kirillov Representation}\label{Kirillov}
Let $F$ be a finite field with $q$ elements and $G=\GL(n,F)$. Let $P_n$ be the Mirabolic subgroup of $G$ and let $U$ be the subgroup of unipotent matrices of $G$. In this section, we recall the Kirillov representation of the Mirabolic subgroup $P_n$ of $G$. Let $\psi_0$ be a non-trivial character of $F$ and let $\psi:U \to \mathbb{C}^{\times}$ be the non-degenerate character of $U$ given by 
\[\psi \left( \begin{bmatrix} 1 & x_{12} & x_{13} & \cdots & x_{1,n}\\
& 1 & x_{23} & \cdots & x_{2,n}\\
& & 1 & \cdots & \vdots \\
& & & \ddots & x_{(n-1),n} \\ 
& & & &1 \end{bmatrix} \right)=\psi_0(x_{12}+x_{23} + \cdots + x_{(n-1),n}).\]
Then, $\mathcal{K}=\ind_{U}^{P_n}\psi$ is called the \textit{Kirillov representation} of $P_n$.\\

\begin{theorem}\label{Kirillov representation} $\mathcal{K}=\ind_{U}^{P_n}\psi$ is an irreducible representation of $P_{n}$. 
\end{theorem}

We refer the reader to Theorem $5.1$ in \cite{Bump} for a proof.

\subsection{Multiplicity one Theorem for $\GL(n,F)$ over a finite field $F$}
We continue with the notation of section \ref{Kirillov}. 
\begin{theorem} Let $\mathcal{G}=\ind_{U}^{G}(\psi)$. The representation $\mathcal{G}$ of $G$ is multiplicity free. 
\end{theorem}

We refer to Theorem $6.1$ in \cite{Bump} for a proof.

\subsection{Twisted Jacquet Module}

In this section, we recall the character and the dimension formula of the twisted Jacquet module of a representation $\pi$. \\

Let $G=\GL(k,F)$ and $P=MN$ be a parabolic subgroup of $G$. Let $\psi$ be a character of $N$. For $m\in M$, let $\psi^{m}$ be the character of $N$ defined by $\psi^{m}(n)=\psi(mnm^{-1})$. Let \[V(N,\psi)= \Span_{\mathbb{C}} \{\pi(n)v-\psi(n)v \mid n \in N, v \in V \} \]
and \[M_{\psi} = \{m \in M \mid {\psi}^{m}(n)=\psi(n) , \forall n \in N \} . \]
Clearly, $M_{\psi}$ is a subgroup of $M$ and it is easy to see that $V(N,\psi)$ is an $M_{\psi}$-invariant subspace of $V$. Hence, we get a representation  $(\pi_{N,\psi},V/V(N,\psi))$ of $M_{\psi}$. We call $(\pi_{N,\psi}, V/V(N,\psi))$ the twisted Jacquet module of $\pi$ with respect to $\psi$. We write $\Theta_{N, \psi}$ for the character of $\pi_{N, \psi}$.

\begin{proposition}
Let $(\pi,V)$ be a representation of $\GL(k,F)$  and $\Theta_\pi$ be the character of $\pi$. We have
\[\Theta_{N,\psi}(m) = \frac{1}{|N|}\sum_{n \in N}\Theta_{\pi}(mn)\overline{\psi(n)}.\]
\end{proposition}
We refer the reader to Proposition 2.3 in \cite{KumHim} for a proof. 

\begin{remark}
Taking $m=1$, we get the dimension of $\pi_{N, \psi}$. To be precise, we have
\[\dim_{\mathbb{C}}(\pi_{N,\psi}) =  \frac{1}{|N|}\sum_{n \in N}\Theta_{\pi}(n)\overline{\psi(n)}.\]
\end{remark}

\subsection{$q$-Hypergeometric Identity}

In this section, we record a certain $q$-identity from \cite{HazGor} which we use in calculating the dimension of the twisted Jacquet module. Before we state it, we set up some notation. Let $\M(n,m,r,q)$ be the set of all $n\times m$ matrices of rank $r$ over the finite field $F$ of order $q$ and $(a;q)_{n}$ be the $q$-Pochhammer symbol defined by 
\[(a;q)_{n}= \prod_{i=0}^{n-1}(1-aq^{i}).\]

\begin{proposition} Let $a$ be an integer greater than or equal to $2n$. Then
\[\sum_{r\geq 0}\M(n,n,r,q)(q;q)_{a-r}= q^{n^{2}}\frac{(q;q)^{2}_{a-n}}{(q;q)_{a-2n}}. \]
\end{proposition}
We refer the reader to Lemma 2.1 in \cite{HazGor} for a proof of the above proposition in a more general set up.

\section{Dimension of the Twisted Jacquet Module}
 
Let $\pi=\pi_{\theta}$ be an irreducible cuspidal representation of $G$ corresponding to the regular character $\theta$ of $F_{2n}^{\times}$ and $\Theta_{\theta}$ be its character. In this section, we calculate the dimension of $\pi_{N,\psi_{A}}$, where 
\[A= \begin{bmatrix} 1 & 0 & \hdots & 0\\
0 &  0 & \hdots  & 0 \\
\vdots & \vdots &  \ddots &\vdots\\
0 & 0 & \hdots & 0
\end{bmatrix}.\]

Throughout, we write $\M(n,m,r,q)$ denote the set of $n\times m$ matrices of rank $r$ over the finite field $F$ of cardinality $q$. For $\alpha \in F$ and $0 \leq r \leq n$, consider the subset $Y_{n,r}^{\alpha}$ of $\M(n,F)$ given by 
\[Y_{n,r}^{\alpha}=\{X \in \M(n,F) \mid \Rank(X)=r, \tr(AX)=\alpha\}.\]


\begin{lemma}\label{M(n,n,r,q)} We have 
$$|\M(n,n,r,q)| = q^r|\M(n,n-1,r,q)| + (q^n-q^{r-1})|\M(n,n-1,r-1,q)|.$$
\end{lemma}

\begin{proof} Let $S= q^r|\M(n,n-1,r,q)| + (q^n-q^{r-1})|\M(n,n-1,r-1,q)|$. It is well known that   
\[ |\M(n,m,r,q)|= \prod_{j=0}^{r-1}\frac{(q^{n}-q^{j})(q^{m}-q^{j})}{(q^{r}-q^{j})}. 
\]
Thus, we have 
\begin{align*}
S &= q^r\prod_{j=0}^{r-1}\frac{(q^{n}-q^{j})(q^{n-1}-q^{j})}{(q^{r}-q^{j})} + (q^n-q^{r-1}) \prod_{j=0}^{r-2}\frac{(q^{n}-q^{j})(q^{n-1}-q^{j})}{(q^{r-1}-q^{j})}\\ \\
&= \prod_{j=0}^{r-1}\frac{(q^{n}-q^{j})(q^{n}-q^{j+1})}{(q^{r}-q^{j})} + (q^n-q^{r-1}) \prod_{j=0}^{r-2}\frac{(q^{n}-q^{j})(q^{n}-q^{j+1})}{(q^{r}-q^{j+1})}\\ \\
&= \frac{q^n-q^r}{q^n-1}|\M(n,n,r,q)| + \frac{(q^n-q^{r-1})(q^r-1)}{(q^n-q^{r-1})(q^n-1)}|\M(n,n,r,q)|\\\\
&= |\M(n,n,r,q)|.
\end{align*}
\end{proof}

\begin{lemma}\label{equal cardinality} Let $r\in \{1,2,3, \dots, n\}$ and $\alpha, \beta \in F^{\times}$. Then we have
\[|Y_{n,r}^{\alpha}| = |Y_{n,r}^{\beta}|. \]
\end{lemma}
\begin{proof} Consider the map $\phi: Y_{n,r}^{\alpha} \to Y_{n,r}^{\beta}$ given by \[\phi(X)=\alpha^{-1}\beta X.\] 
Suppose that $\phi(X)=\phi(Y)$. Since $\alpha^{-1}\beta \neq 0$, it follows that $\phi$ is injective. For $Y \in Y_{n,r}^{\beta}$, let $X=\alpha \beta^{-1} Y$. Clearly, we have $\tr(AX)=\alpha$ and $\Rank(X)=\Rank(Y)=r$. Thus $\phi$ is surjective and hence the result.
\end{proof}

\begin{lemma}\label{cardinality of trace 0} 
$$|Y^{0}_{n,r}|= q^{-1}|\M(n,n,r,q)| + (q^r - q^{r-1})|\M(n-1,n-1,r,q)| + (q^{r-2}-q^{r-1})|\M(n-1,n-1,r-1,q)|. $$
\end{lemma}

\begin{proof} Let $\mathfrak{B} =\{e_1,e_2,\hdots,e_n\}$ be a basis of $F^{n}$ over $F$ and $X\in Y_{n,r}^{0}$. Then, \[[X]_{\mathfrak{B}}=\begin{bmatrix} 0 & w \\ v & Y\end{bmatrix}\] where $w$ is an $1 \times (n-1)$ row vector, $v$ is an $(n-1) \times 1$ column vector and $Y$ is an $(n-1)\times (n-1)$ block matrix. We also write \[ \begin{bmatrix} w \\ Y\end{bmatrix} =   \begin{bmatrix} v_1 & v_2 & \cdots & v_{n-1} \end{bmatrix}\] where $v_i$ is an $n \times 1 $ column vector for $1 \leq i \leq n-1$.\\

Let $V$ be the $n-1$ dimensional hyperplane spanned by the vectors $\{e_2,e_3,\dots, e_n\}$. It is easy to see that $\begin{bmatrix}0\\ v \end{bmatrix} \in V$. We let $W$ be the space spanned by the vectors $\{v_1,v_2,\dots, v_{n-1}\}$. Since $X \in Y_{n,r}^{0}$, the rank of the $n\times (n-1)$ matrix \[\begin{bmatrix} w \\ Y\end{bmatrix} =  \begin{bmatrix} v_1 & v_2 & \cdots & v_{n-1} \end{bmatrix}\] has only two possibilities, either $r$ or $r-1$.
We consider both these cases separately. \\

\begin{enumerate}
\item[Case 1)] Suppose that \[\Rank \left( \begin{bmatrix} w \\ Y \end{bmatrix}\right) = \Rank(\begin{bmatrix} v_1 & v_2 & \cdots & v_{n-1} \end{bmatrix})=r.\] Then $\dim W=r$. It follows that, $\begin{bmatrix}0 \\ v \end{bmatrix} \in W$ and hence $\begin{bmatrix} 0\\ 
v \end{bmatrix} \in V \cap W$. Therefore, the number of choices for $\begin{bmatrix} v_1 & v_2 & \cdots & v_{n-1} \end{bmatrix}$ is $|\M(n,n-1,r,q)|$.  \\ 
\item[a)] If $w=0$, then \[W \subseteq V.\] Hence, $V\cap W=W$ and $\dim(V\cap W) = \dim W = r.$ Since $\begin{bmatrix}0 \\ v \end{bmatrix} \in V \cap W$, the number of possibilities of $\begin{bmatrix}0 \\ v \end{bmatrix}$ will be $q^r$. Also, the total number of matrices $\begin{bmatrix} w \\ Y\end{bmatrix}$ with rank $r$ and $w=0$ is $|\M(n-1,n-1,r,q)|$.\\

\item[b)] If $w \neq 0$, we have $W\not \subseteq V$. Therefore, 
\begin{align*}\dim (W\cap V) &= \dim V + \dim W -\dim(V+W) \\
&= n-1 + r -n = r-1.
\end{align*} 
Since $\begin{bmatrix}0 \\ v \end{bmatrix}\in V\cap W$, the number of possibilities of $\begin{bmatrix}0 \\ v \end{bmatrix}$ will be $q^{r-1}$. The number of matrices $\begin{bmatrix} w \\ Y\end{bmatrix}$ with rank $r$ and $w \neq 0$, is \[|\M(n,n-1,r,q)|-|\M(n-1,n-1,r,q)|.\]
\end{enumerate}

\begin{enumerate}
\item[Case 2)] Suppose that \[\Rank \left( \begin{bmatrix} w \\ Y\end{bmatrix}\right)=\Rank(\begin{bmatrix} v_1 & v_2 & \cdots & v_{n-1} \end{bmatrix})=r-1.\] Then $\dim W =r-1$. Therefore, $v \not \in W$ and hence $\begin{bmatrix}0 \\ v \end{bmatrix} \in V \backslash W$. Also,we have that the total number of matrices $ \begin{bmatrix} w \\ Y\end{bmatrix}$ with rank $r-1$ is $|\M(n,n-1,r-1,q)|.$ \\  
\item[a)] If $w=0$, then $W \subseteq V$. Therefore, $V\cap W=W$ and $\dim (V\cap W) = \dim W = r-1$. Since $\begin{bmatrix}0 \\ v \end{bmatrix}\in V\backslash W$, the number of possibilities of $\begin{bmatrix}0 \\ v \end{bmatrix}$ will be $q^{n-1} - q^{r-1}$. Furthermore, The total number of matrices $\begin{bmatrix} w \\ Y\end{bmatrix}$ with rank $r-1$ and $w=0$ is $|\M(n-1,n-1,r-1,q)|$.\\

\item[b)] If $w\neq 0$, then $W\not \subseteq V$. Therefore, 
\begin{align*}
\dim (W\cap V) &= \dim V + \dim W -\dim(V+W)\\
&= n-1 + r-1 -n = r-2.
\end{align*} 
Since $v\in V\backslash W$, the number of possibilities of $\begin{bmatrix}0 \\ v \end{bmatrix}$ will be $q^{n-1} - q^{r-2}$. The total number of matrices in this case will be $|\M(n,n-1,r-1,q)|-|\M(n-1,n-1,r-1,q)|$.\\
\end{enumerate}
Using Lemma \ref{M(n,n,r,q)}, and the above computations, we have

\begin{align*} 
|Y^{0}_{n,r}| &= q^r|\M(n-1,n-1,r,q)| + q^{r-1}(|\M(n,n-1,r,q)|-|\M(n-1,n-1,r,q)|) \\
          &+(q^{n-1} - q^{r-1})|\M(n-1,n-1,r-1,q)|\\
          &+(q^{n-1} - q^{r-2})(|\M(n,n-1,r-1,q)|-|\M(n-1,n-1,r-1,q)|)\\\\           
          &= q^{r-1}|\M(n,n-1,r,q)| + (q^{n-1} - q^{r-2})|\M(n,n-1,r-1,q)|\\
          &+ (q^r - q^{r-1})|\M(n-1,n-1,r,q)|\\
          &+ (q^{n-1}-q^{r-1}-q^{n-1}+q^{r-2})|\M(n-1,n-1,r-1,q)|\\ \\          
          &= q^{r-1}|\M(n,n-1,r,q)| + (q^{n-1} - q^{r-2})|\M(n,n-1,r-1,q)|\\
          &+ (q^r - q^{r-1})|\M(n-1,n-1,r,q)| + (q^{r-2}-q^{r-1})|\M(n-1,n-1,r-1,q)|\\ \\          
          &= q^{-1}|\M(n,n,r,q)|+ (q^r - q^{r-1})|\M(n-1,n-1,r,q)|\\
          &+ (q^{r-2}-q^{r-1})|\M(n-1,n-1,r-1,q)|. 
\end{align*}
\end{proof}

\begin{lemma}\label{cardinality of trace 1} We have 
\[|Y^{1}_{n,r}| = q^{-1}|\M(n,n,r,q)|-q^{r-1}|\M(n-1,n-1,r,q)|+q^{r-2}|\M(n-1,n-1,r-1,q)|.\]
\end{lemma}
\begin{proof} Using Lemma~\ref{equal cardinality}, we have 
$$|Y^{0}_{n,r}|+(q-1)|Y^{1}_{n,r}|=|\M(n,n,r,q)|.$$ Thus we get,
\begin{align*}
    |Y^{1}_{n,r}| &= \frac{\M(n,n,r,q)| - |Y^{0}_{n,r}|}{q-1}\\\\
    &= \frac{\M(n,n,r,q)| - q^{-1}|\M(n,n,r,q)|}{q-1}\\
    &-\frac{(q^r - q^{r-1})|\M(n-1,n-1,r,q)|+ (q^{r-2}-q^{r-1})|\M(n-1,n-1,r-1,q)|}{q-1}\\\\
    &= q^{-1}|\M(n,n,r,q)| - q^{r-1}|\M(n-1,n-1,r,q)| + q^{r-2}|\M(n-1,n-1,r-1,q)|.
\end{align*}
\end{proof}

\begin{lemma} We have 
\[|Y^{0}_{n,r}|-|Y^{1}_{n,r}|= q^r|\M(n-1,n-1,r,q)| - q^{r-1}|\M(n-1,n-1,r-1,q)|.\]
\end{lemma}

\begin{proof} Follows from Lemma~\ref{cardinality of trace 0} and Lemma~\ref{cardinality of trace 1}.  
\end{proof}

\begin{lemma}
Let $r \in\{0,1,2,\dots, n\}$ and $X \in \mathrm{M}(n,n, r, q)$. We have
$$
\Theta_{\theta}\left(\left[\begin{array}{cc}
1 & X \\
0 & 1
\end{array}\right]\right)
=
\begin{cases}
(-1)(q;q)_{2n-1}, & \text { if } r=0 \\
(-1)(q;q)_{2n-2}, & \text { if } r=1 \\
\quad \vdots \\
(-1)(q;q)_{n-1}, & \text { if } r=n
\end{cases}
$$
\end{lemma}

\begin{proof} The proof follows from Theorem~\ref{character value of cuspidal representation (Dipendra)} above and rewriting the character values using the $q$-Pochhammer symbol.  
\end{proof}

\begin{theorem}\label{dimension calculation} Let $\theta$ be a regular character of $F_{2n}^{\times}$ and $\pi=\pi_{\theta}$ be an irreducible cuspidal representation of $\GL(2n,F)$. We have \[\dim_{\mathbb{C}}(\pi_{N, \psi_{A}})= (q-1)^{2}(q^2-1)^{2}\cdots (q^{n-1}-1)^{2}=(q;q)^{2}_{n-1}.\]
\end{theorem}

\begin{proof} It is easy to see that the dimension of $\pi_{N,\psi_{A}}$ is given by 
\[\dim_{\mathbb{C}}(\pi_{N,\psi_{A}}) = \frac{1}{q^{n^2}} \sum_{X\in \M(n,F)}\Theta_{\theta}\left(\begin{bmatrix} 1 & X \\ 0 & 1\end{bmatrix}\right)\overline{\psi_{0}(\tr(AX))}. \]
Clearly, we have $\displaystyle \M(n,F)= \bigcup_{r=0}^{n}\left(\bigcup_{\alpha\in F}Y^{\alpha}_{n,r}\right)$. Using this, we see that 
\begin{align*}
\dim_{\mathbb{C}}(\pi_{N,\psi_{A}}) &= \frac{1}{q^{n^2}} \sum_{r=0}^n \sum_{\substack{X\in Y^{\alpha}_{n,r} \\ \alpha \in F}}\Theta_{\theta}\left(\begin{bmatrix} 1 & X \\ 0 & 1\end{bmatrix}\right)\overline{\psi_{0}(\alpha)}\\
&= \frac{1}{q^{n^2}} \sum_{r=0}^n (-1)(q;q)_{2n-1-r}\left(|Y^{0}_{n,r}|-|Y^{1}_{n,r}|\right)\\
&= -\frac{1}{q^{n^2}} \big[(q;q)_{2n-1}q^0 \left|\M(n-1,n-1,0,q)\right | \\
&+  \sum_{r=1}^n (q;q)_{2n-1-r}(q^r \left|\M(n-1,n-1,r,q)\right| - q^{r-1}\left|\M(n-1,n-1,r-1,q)\right|)\big]\\
&= -\frac{1}{q^{n^2}} \sum_{r=0}^{n-1} q^r \left((q;q)_{2n-1-r} - (q;q)_{2n-2-r}\right) \left|\M(n-1,n-1,r,q)\right|\\ 
&= \frac{1}{q^{n^2}}\sum_{r=0}^{n-1} q^{2n-1}\left|\M(n-1,n-1,r,q)\right|(q;q)_{2n-2-r}\\ 
&= \frac{1}{q^{(n-1)^2}}\sum_{r=0}^{n-1} \left|\M(n-1,n-1,r,q)\right|(q;q)_{2n-2-r}\\
&= (q;q)_{n-1}^2.
\end{align*}
\end{proof}

\begin{remark} Suppose that $B=Aw_{0}$. It is easy to see that $\Theta_{N,\psi_{A}}\left(\begin{bmatrix} m_{1} & 0 \\ 0 & m_{2}\end{bmatrix}\right)= \Theta_{N,\psi_{B}}\left( \begin{bmatrix} w_{0}m_{1}w_{0} & 0 \\ 0 & m_{2}\end{bmatrix} \right)$. Thus we have that $\dim_{\mathbb{C}}(\pi_{N,\psi_{A}})=\dim_{\mathbb{C}}(\pi_{N, \psi_{B}})$. \\
\end{remark}



\section{Main Theorem}
In this section, we prove the main result of this paper. Before we continue, we set up some notation and record a few preliminary results that we need. 
Let $G = \GL(2n,F)$ and $P$ be the maximal parabolic subgroup of $G$ with Levi decomposition $P = MN$, where $M \simeq \GL(n,F) \times \GL(n,F)$ and $N \simeq \M(n,F)$. We write $F_{n}$ for the unique field extension of $F$ of degree $n$. Let $\psi_{0}$ be a fixed non-trivial additive character of $F$.
Let \[A= \begin{bmatrix} 0 & \hdots & 0 & 1\\
0 & \hdots  & 0 &  0 \\
\vdots & \ddots &  \vdots &\vdots\\
0 & \hdots & 0 & 0
\end{bmatrix}.\]
Let $\psi_{A}: N \rightarrow \mathbb{C}^{\times}$ be the character of $N$ given by 
\[\psi_{A} \left( \begin{bmatrix} 1 & X\\
0 & 1 \end{bmatrix} \right)=\psi_{0}(\tr(AX)).\] 
Let $H_{A}=M_{1} \times M_{2}$ where $M_{1}$ is the Mirabolic subgroup of $\GL(n,F)$ and $M_{2}={w_0}{M_{1}}^{\top}{w_0}^{-1}$. Let $U$ be the subgroup of unipotent matrices in $\GL(2n,F)$. Let $U_{A}=H_{A} \cap U$. Clearly, we have $U_{A} \simeq U_{1} \times U_{2}$ where $U_{1}$ and
$U_{2}$ are the upper triangular unipotent subgroups of $\GL(n, F )$. For $k=1,2$, let
$\mu_{k} : U_{k} \rightarrow \mathbb{C}^{\times}$ be the non-degenerate character of $U_{k}$ given by
\[\mu_{k} \left ( \begin{bmatrix} 
1 & x_{12} & x_{13} & \hdots & x_{1n} \\
 & 1 & x_{23} & \hdots & x_{2n}\\
 & & 1 & \ddots & \vdots \\
 & & & \ddots & x_{n-1,n} \\
 & & & & 1 \end{bmatrix} \right) =\psi_0(x_{12}+x_{23}+ \cdots + x_{n-1,n}).\] 
 Let $\mu:U_{A} \rightarrow \mathbb{C}^{\times}$ be the character of $U_{A}$ given by 
 \[\mu(u)=\mu_{1}(u_1)\mu_{2}(u_2)\]
 where $u=\begin{bmatrix} u_{1} & \\
 & u_{2} \end{bmatrix}$.

\begin{lemma}\label{elements in M psi A} Let $M_{\psi_{A}}=\{m\in M \mid \psi^{m}_{A}(n)= \psi_{A}(n), \forall n\in N\}$. Then we have
\[M_{\psi_{A}}= \left\lbrace\begin{bmatrix} C & x & & \\
0 & a & &\\
& & a & y\\
& & 0 & D \end{bmatrix} \,\middle\vert\, a \in F^{\times},C, D \in \GL(n-1,F), x, y \in F^{n-1} \right\rbrace.\]
\end{lemma}

\begin{proof} Let $g=\begin{bmatrix} g_1 & \\
& g_2 \end{bmatrix} \in M$. Then $ g \in M_{\psi_A}$ if and only if $ Ag_1=g_2A$. It follows that $g \in M_{\psi_A}$ if and only if $g_1=\begin{bmatrix}C & x \\
0 & a\end{bmatrix}$ and $g_2=\begin{bmatrix}
  a &  y\\
  0 & D  \end{bmatrix}$.
\end{proof}

\begin{lemma}
Let $Z$ be the center of $G=\GL(2n,F)$. Let $H_{A}$ be a subgroup of $G$ as above. Then,
\[M_{\psi_{A}} \simeq Z \times H_{A}.\]
\end{lemma}

\begin{proof} Trivial.
\end{proof}

\begin{lemma} \label{irred rho M} Let $\rho_1=\ind_{U_1}^{M_1}\mu_1$ and  $\rho_2=\ind_{U_2}^{M_2}\mu_2$. Consider the representation $(\rho,V)$ of $M_{\psi_{A}}$ given by  
\[\rho=\theta|_{F^{\times}} \otimes \ind_{U_A}^{H_A}\mu=\theta|_{F^{\times}} \otimes (\rho_1 \otimes \rho_2). \]
Then $(\rho,V)$ is an irreducible representation of $M_{\psi_A}$.
\end{lemma}

\begin{proof} Since $\rho_1$ is the Kirillov representation of the Mirabolic subgroup $M_1$ of $\GL(n,F)$, we have that $\rho_1$ is irreducible (see Theorem~\ref{Kirillov representation}). In a similar way, we can see that $\rho_2$ is also irreducible. Hence the result. 
\end{proof}

\begin{lemma}
Let $P_{\psi_A}=M_{\psi_A}N$. Consider the map $\tilde{\rho}:P_{\psi_A} \to \GL(V)$ given by 
\[\tilde{\rho}(p)=\tilde{\rho}(mn)=\psi_A(mn{m}^{-1})\rho(m),\]
where $m\in M_{\psi_{A}}, n\in N$. Then $(\tilde{\rho},V)$ is a representation of $P_{\psi_A}$.  
\end{lemma}

\begin{proof} Let $p_1=m_1n_1,p_2=m_2n_2 \in P_{\psi_A}$. Then, we have 
\begin{align*}
\tilde{\rho}(p_1p_2)&=\tilde{\rho}(m_1n_1m_2n_2)\\
&=\tilde{\rho}(m_1m_2({m_2}^{-1}n_1m_2)n_2)\\
&=\psi_A(m_1m_2({m_2}^{-1}n_1m_2n_2){m_2}^{-1}{m_1}^{-1})\rho(m_1m_2)\\
&=\psi_A(n_1(m_2n_2{m_2}^{-1}))\rho(m_1m_2)\\
&=\psi_A(n_1)\psi_A(m_2n_2{m_2}^{-1})\rho(m_1)\rho(m_2)\\
&=\psi_A(m_1{n_1}{m_1}^{-1})\rho(m_1)\psi_A(m_2n_2{m_2}^{-1})\rho(m_2)\\
&=\tilde{\rho}(p_1)\tilde{\rho}(p_2).
\end{align*}

\end{proof}

\begin{lemma}
Let $(\tilde{\rho},V)$ be the representation of $P_{\psi_A}$ given by 
\[\tilde{\rho}(p)=\tilde{\rho}(mn)=\psi_A(mn{m}^{-1})\rho(m),\]
where $m\in M_{\psi_{A}}, n\in N$. Then, 
$(\tilde{\rho},V)$ is irreducible.
\end{lemma}

\begin{proof} Let $W$ be a non-trivial $P_{\psi_A}$-invariant subspace of $V$. For $w\in W, p\in P_{\psi_{A}}$, we have 
\[\tilde{\rho}(p)w=\psi_A(mnm^{-1})\rho(m)w\in W. \]
Therefore $\rho(m)w \in W$, for all $m\in M_{\psi_{A}}$, $w\in W$. Since $\rho$ is irreducible (see Lemma~\ref{irred rho M}), the result follows. 
\end{proof}

\begin{lemma} Consider the representation $\tilde{\rho}$ of $P_{\psi_{A}}$. We have  
\[\tilde{\rho}|_{U}=\psi_{A} \otimes \rho|_{U_{A}}.\]
\end{lemma}
 
\begin{proof} Clearly we have $U=U_{A}N$. Hence for $u=xn\in U$, we have 
\[\tilde{\rho}(u)=\psi_{A}(xn{x}^{-1})\rho(x)=\psi_{A}(n)\rho(x).\]
\end{proof}

\begin{lemma}\label{equality of central characters} Let $\rho=\theta|_{F^{\times}} \otimes \ind_{U_{A}}^{H_{A}}\mu$ be the representation of $M_{\psi_{A}}$ and $\tilde{\rho}$ be the corresponding representation of $P_{\psi_{A}}$. For any $z \in Z$, we have 
\[\omega_{\tilde{\rho}}(z)= \omega_{\rho}(z)=\theta(z).\]
\end{lemma}

\begin{proof} For $z \in Z$, we have 
\begin{align*}
\chi_{\tilde{\rho}}(z) &= \tr(\tilde{\rho}(z)) \\
&= \omega_{\tilde{\rho}}(z)\deg(\rho) \\
&= \tr(\rho(z))\\
&= \omega_{\rho}(z)\deg(\rho)\\
&= \tr(\theta|_{F^{\times}}(z)\otimes \ind_{U_{A}}^{H_{A}}\mu(1))\\
&= \theta(z)\deg(\rho) 
\end{align*}
It follows that $\omega_{\tilde{\rho}}(z)=\omega_{\rho}(z)=\theta(z)$.
\end{proof}

\begin{lemma}
Let $\chi: F^{\times} \to \mathbb{C}^{\times}$ be a character of $F^{\times}$. Consider the representation $(\tilde{\rho},V)$ of $P_{\psi_A}$ defined above. Let 
$\sigma_{\chi}:P_{\psi_A} \to \GL(V)$ be the map  
\[\sigma_{\chi}(p)= \sigma_{\chi}(zhn)=\chi(z)\tilde{\rho}(hn),\]
where $z\in Z, h\in H_{A}, n\in N$. Then $\sigma_{\chi}$ is an irreducible representation of $P_{\psi_A}$. 
\end{lemma}

\begin{proof} It is easy to see that $\sigma_{\chi}$ is a representation of $P_{\psi_A}$. Let $W$ be a non-trivial subspace of $V$ invariant under $P_{\psi_A}$ and let $w\neq 0 \in W$. We have  
\[\sigma_{\chi}(zhn)w = \chi(z)\tilde{\rho}(hn)w \in W. \]
Therefore, 
\[\tilde{\rho}(zhn)w=\tilde{\rho}(z)\tilde{\rho}(hn)w=\omega_{\tilde{\rho}}(z)\tilde{\rho}(hn)w \in W.\]
Since $\tilde{\rho}$ is irreducible, it follows that $V=W$ and hence the result. 
\end{proof}

\begin{lemma}\label{sigma chi are inequivalent}
Let $\chi_1, \chi_2 \in \widehat{F^{\times}}$ such that $\chi_1 \neq \chi_2$.
Then, \[\sigma_{\chi_1} \not \simeq \sigma_{\chi_2}.\] 
\end{lemma}

\begin{proof}
Let $z_0 \in Z$ such that $\chi_1(z_0)\neq \chi_2(z_0)$. Let $\chi_{\sigma_{\chi_{1}}}$, $\chi_{\sigma_{\chi_{2}}}$ be the characters of $\sigma_{\chi_{1}}$ and $\sigma_{\chi_{2}}$. Suppose that $\sigma_{\chi_1} \simeq \sigma_{\chi_2}$. We have 
\begin{align*}\label{sigma(chi)} 
\chi_{\sigma_{\chi_1}}(z_0) &= \tr(\sigma_{\chi_{1}}(z_{0}))\\
&= \chi_{1}(z_{0})\deg(\rho)\\
&= \chi_{\sigma_{\chi_2}}(z_0)\\
&= \tr(\sigma_{\chi_{2}}(z_{0}))\\
&=\chi_{2}(z_{0})\deg(\rho).\\
\end{align*}  
The result follows. 
\end{proof}

\begin{lemma}\label{sigma chi occurs in the induced representation}
For $\chi \in \widehat{F^{\times}}$, we have 
\[\Hom_{P_{\psi_A}}(\sigma_{\chi}, \ind_{U}^{P_{\psi_A}}\psi) \neq 0.\]
\end{lemma}

\begin{proof} Using Fr\"{o}benius Reciprocity, we have 
\[\Hom_{P_{\psi_A}}(  \sigma_{\chi}, \ind_{U}^{P_{\psi_A}}\psi)=\Hom_{U}(\sigma_{\chi}|_{U}, \psi). \]
Thus it is enough to show that $\Hom_{U}(\sigma_{\chi}|_{U}, \psi)\neq 0$. For $u \in U$, we have 
\[\sigma_{\chi}|_{U}(u)=\sigma_{\chi}(u)=\chi(1)\tilde{\rho}(u)=\tilde{\rho}|_U(u).\]
Therefore,
\begin{align*}
\Hom_{U}(\sigma_{\chi}|_U, \psi)&=\Hom_U(\tilde{\rho}|_U, \psi)\\\\
&= \Hom_U(\psi_A \otimes \rho|_{U_A},\psi)\\\\
&= \Hom_{U}\left(\psi_A \otimes \bigoplus_{s \in U_A \setminus H_A / U_A}\ind_{s^{-1}U_{A}s \cap U_A}^{U_A}\mu^{s}, \psi\right)\\\\
&= \Hom_{U}(\psi_A \otimes \mu, \psi) \oplus \bigoplus_{1\neq s \in U_A \setminus H_A / U_A} \Hom_{U}\left(\psi_{A}\otimes \ind_{s^{-1}U_{A}s \cap U_A}^{U_A}\mu^{s}, \psi\right )\\\\
&= \Hom_{U}(\psi, \psi) \oplus \bigoplus_{1\neq s \in U_A \setminus H_A / U_A} \Hom_{U}\left(\psi_{A}\otimes \ind_{s^{-1}U_{A}s \cap U_A}^{U_A}\mu^{s}, \psi\right )\\\\
&\neq 0.
\end{align*}
\end{proof}

\begin{lemma}\label{Explanation of the induced representation} Let $\chi\in \widehat{F^{\times}}$ and $\sigma_{\chi}$ be the irreducible representation of $P_{\psi_{A}}$. Then  
\[\ind_{U}^{P_{\psi_A}}\psi= \bigoplus_{\chi \in \widehat{F^{\times}}}\sigma_{\chi}.\]
\end{lemma}

\begin{proof} The result clearly follows from a simple application of Lemma~\ref{sigma chi are inequivalent} and Lemma~\ref{sigma chi occurs in the induced representation}, and computing the degree of $\ind_{U}^{P_{\psi_A}}(\psi)$. To be precise, suppose that 
\[\ind_{U}^{P_{\psi_A}}(\psi)=(\bigoplus_{x \in \widehat{F^{\times}}}d_{\chi}\sigma_{\chi}) \oplus d\sigma\]
where $d_{\chi} \geq 1, d \geq 0$ and $\sigma$ is some representation of $P_{\psi_A}$. By degree comparison, we have that 
\[\deg\big(\bigoplus_{\chi \in \widehat{F^{\times}}}d_{\chi}\sigma_{\chi}\big) =\sum_{\chi \in \widehat{F^{\times}}}d_{\chi}\deg(\sigma_{\chi})=\sum_{\chi \in \widehat{F^{\times}}}d_{\chi}\deg(\rho) 
\]
Clearly \[\sum_{\chi \in \widehat{F^{\times}}}d_{\chi}\deg(\rho)  \geq (q-1)\deg(\rho)=\deg(\ind_{U}^{P_{\psi_A}}(\psi)),\]
On the other hand, we have 
\[\deg(\bigoplus_{\chi \in \widehat{F^{\times}}}d_{\chi}\sigma_{\chi})+d\deg(\sigma) = \deg(\ind_{U}^{P_{\psi_A}}(\psi)).\]
It follows that 
\[d=0,d_{\chi}=1, \forall \chi \in F^{\times}.\]
Hence the result.
\end{proof}

\begin{lemma} Let $m=ah \in M_{\psi_{A}}$, where $a\in Z$ and $h\in H_{A}$. Then,
\[\Theta_{N,\psi_A}(m)=\theta(a)\Theta_{N,\psi_A}(h).\]
\end{lemma}

\begin{proof} We have
\begin{align*}
\Theta_{N,\psi_{A}}(m) &=\Theta_{N,\psi_{A}}(ah)\\
&= \frac{1}{|N|}\sum_{n \in N}\Theta_{\theta}(ahn)\overline{\psi_{A}(n)}\\
&=\frac{1}{|N|}\sum_{n \in N}\tr(\pi(ahn)\overline{\psi_A(n)}\\
&=\frac{1}{|N|}\sum_{n \in N}\tr(\pi(a)\pi(hn)\overline{\psi_A(n)}\\
&=\omega_{\pi}(a)\frac{1}{|N|}\sum_{n \in N}\tr(\pi(hn)\overline{\psi_A(n)}\\
&=\omega_{\pi}(a)\Theta_{N,\psi_A}(h)
\end{align*}
where $\omega_{\pi}$ is the central character of $\pi$. Explicitly, we have \[\Theta_{\theta}(a)=\tr(\pi(a))=\tr(\omega_{\pi}(a))=\omega_{\pi}(a)\dim(\pi).\]
Using Theorem \ref{character value of cuspidal representation (Dipendra)}, it is easy to see that $$\Theta_{\theta}(a)=\theta(a)\dim(\pi).$$
Thus, we have $\omega_{\pi}(a)=\theta(a)$ and the result follows. 
\end{proof}

\begin{lemma} Let $\chi\neq \theta \in \widehat{F^{\times}}$. Then 
\[\Hom_{P_{\psi_A}}(\pi|_{P_{\psi_A}},\sigma_{\chi})=0.\]
\end{lemma}

\begin{proof} It is enough to show that $\dim_{\mathbb{C}}\Hom_{P_{\psi_A}}(\pi|_{P_{\psi_A}},\sigma_{\chi})=0$. Clearly, we have 
\begin{align*} 
\dim_{\mathbb{C}}\Hom_{P_{\psi_A}}(\pi|_{P_{\psi_A}},\sigma_{\chi}) &= \langle \chi_{ \pi|_{P_{\psi_A}}}, \chi_{\sigma_{\chi}} \rangle \\
&=\sum_{zhn \in P_{\psi_A}}\chi_{\pi}(zhn)\overline{\chi_{\sigma_{\chi}}(zhn)}\\
&=\sum_{hn \in H_AN}\sum_{z \in Z}\omega_{\pi}(z)\chi_{\pi}(hn)\overline{\chi(z)} \overline{\chi_{\tilde{\rho}}(hn)}\\
&=\sum_{hn \in H_AN}\sum_{z \in Z}\theta(z)\overline{\chi(z)}\chi_{\pi}(hn)\overline{\chi_{\tilde{\rho}}(hn)}\\
&=\sum_{z \in Z}\theta(z)\overline{\chi(z)} \sum_{hn \in H_AN}\chi_{\pi}(hn)\overline{\chi_{\tilde{\rho}}(hn)}\\
&=\langle \theta, \chi \rangle \sum_{hn \in H_AN}\chi_{\pi}(hn)\overline{\chi_{\tilde{\rho}}(hn)}\\
&=0
\end{align*}
It follows that 
\[\Hom_{P_{\psi_A}}(\pi|_{P_{\psi_A}},\sigma_{\chi})={0} ,\forall \chi \in \widehat{F^{\times}}, \chi \neq \theta.\]
\end{proof}

\begin{lemma} Consider the restriction $\theta|_{F^{\times}}$ of the regular character $\theta$. Then 
\[\sigma_{\theta}= \tilde{\rho}\]
as $P_{\psi_{A}}$ representations. 
\end{lemma}

\begin{proof} Using Lemma~\ref{equality of central characters} we have $\omega_{\tilde{\rho}}(z)=\theta(z)$. Thus for $p=zhn\in P_{\psi_{A}}$, we have 
\begin{align*}
\sigma_{\theta}(zhn)& =\theta(z)\rho(hn)\\
&= \omega_{\tilde{\rho}}(z)\tilde{\rho}(hn)\\
&= \tilde{\rho}(zhn).\\
\end{align*}
\end{proof}


\subsection{Proof of the Main Theorem}

For the sake of completeness, we recall the statement below. 
 
\begin{theorem}
Let $\theta$ be a regular character of $F_{2n}^{\times}$ and $\pi=\pi_{\theta}$ be an irreducible cuspidal representation of $G$. Then 
\[\pi_{N,\psi_A} \simeq \theta|_{F^{\times}} \otimes \ind_{U_{A}}^{H_{A}}\mu\]
as $M_{\psi_A}$ modules.
\end{theorem}

\begin{proof}
Using transitivity of induction and Lemma~\ref{Explanation of the induced representation}, we have that
\begin{align*}
\Hom_{G}(\pi, \ind_{U}^{G}\psi)&=\Hom_{G}(\pi, \ind_{P_{\psi_A}}^{G}(\ind_{U}^{P_{\psi_A}}\psi))\\
&=\Hom_G(\pi, \ind_{P_{\psi_A}}^G(\bigoplus_{\chi \in \widehat{F^{\times}}}\sigma_{\chi}))\\
&=\bigoplus_{\chi \in \widehat{F^{\times}}}\Hom_G(\pi, \ind_{P_{\psi_A}}^G \sigma_{\chi})\\
&=\bigoplus_{\chi \in \widehat{F^{\times}}}\Hom_{P_{\psi_A}}(\pi|_{P_{\psi_A}}, \sigma_{\chi})\\
&=\Hom_{P_{\psi_A}}(\pi|_{P_{\psi_A}}, \sigma_{\theta})\oplus \bigoplus_{\theta\neq \chi \in \widehat{F^{\times}}}\Hom_{P_{\psi_A}}(\pi|_{P_{\psi_A}}, \sigma_{\chi})\\
&=\Hom_{P_{\psi_A}}(\pi|_{P_{\psi_A}}, \tilde{\rho})
\end{align*}
Hence, 
\[ 
\Hom_{G}(\pi, \ind_{U}^{G}(\psi))=\Hom_{P_{\psi_A}}(\pi|_{P_{\psi_A}}, \tilde{\rho}) \simeq \Hom_{G}(\pi,\ind_{P_{\psi_A}}^{G}\tilde{\rho})\simeq \Hom_{M_{\psi_A}}(\pi_{N,\psi_A},\rho).
\]
Using the multiplicity one theorem for $\GL(n)$ (give precise statement in preliminaries), we conclude that
\[\dim_{\mathbb{C}}\Hom_{M_{\psi_A}}(\pi_{N,\psi_A},\rho)=1\]
and it follows that $$\pi_{N,\psi_{A}}\simeq \rho$$
as $M_{\psi_{A}}$ representations.  
\end{proof}

\section*{Acknowledgements}

We thank Professor Dipendra Prasad for suggesting this problem and for some helpful discussions. Research of Kumar Balasubramanian is supported by the SERB grant: MTR/2019/000358.

\bibliographystyle{amsplain}
\bibliography{Twisted-Jacquet-Module}
\end{document}